\documentclass{amsart}
\author{Christian Frank}
\author{A. Salch}
\usepackage{shortsalch}
\usepackage{xy, tikz-cd}
\xyoption{all}
\title{CW-complexes in the Category of Small Categories}
\date{}
\DeclareMathOperator{\CW}{{\rm CW}}
\begin{document}
\begin{abstract}
We compute the collection of CW-complexes in the model category of small categories constructed by Joyal and Tierney. More generally, if $X$ is a connected topological space, we show that the homotopy category of CW-complexes in Joyal-Tierney's model category of sheaves of sets on $X$ is equivalent to the homotopy category of groupoids. As an application of the ideas, we show that the algebraic $K$-theory groups of the category of pointed small categories are trivial, and more generally, the algebraic $K$-theory groups of any sufficiently ``nice'' Waldhausen category $\mathcal{A}$ of pointed small categories also vanishes, regardless of finiteness conditions assumed on the objects of $\mathcal{A}$. The vanishing of this $K$-theory implies that there is no nontrivial Euler characteristic defined on pointed small categories and satisfying certain niceness axioms.

\end{abstract}
\maketitle

\section{Introduction}

In the category of unpointed spaces, the suspension of a space $X$ is weakly equivalent to the homotopy pushout of the diagram 
\begin{center}
 \begin{tikzcd}
  X \arrow{d} \arrow{r} &\pt\\
  \pt
 \end{tikzcd}
\end{center}
where $\pt$ is the terminal object, i.e., the one-point space.
As the $(n + 1)$-sphere $S^{n + 1}$ is the suspension of the $n$-sphere $S^n$, this gives a purely model-categorical method of inductively constructing spaces weakly equivalent to spheres, starting from the zero-sphere $S^0 = \pt \coprod \pt$, a coproduct of two copies of the terminal object.  Homotopy pushouts also provide a model-categorical method of attaching cells, and consequently, of building CW-complexes: to glue an $(n + 1)$-disc to a space $X$ by a map $f: S^n\to X$, we form the diagram
\begin{center}
 \begin{tikzcd}
  S^n \arrow{d} \arrow{r}{f} &X\\
  \pt
 \end{tikzcd}
\end{center}
and take its homotopy pushout. 
It is natural to ask how much of the classical theory of CW-complexes can be generalized from the setting of topological spaces to other model categories of fundamental importance.

In this paper, we examine the formal CW-complexes that can be built when using the {\em Joyal-Tierney model structure,} which we now explain. There exists a model structure on the category $\SmCat$ of small categories in which the weak equivalences are simply the familiar categorical equivalences, and the cofibrations are the functors which are injective on objects; this model structure is a special case of the general {\em Joyal-Tierney model structure}, from~\cite{MR1097495}, on internal categories in a Grothendieck topos, whose fibrant objects form a certain model for stacks (see~\cite{MR3274499} for more results on that connection). It is a folklore result\footnote{To the best of our knowledge, this result was never published, but several writeups including proofs are available on the Web, for example~\cite{joyaltierneyuniqueness}.} that there exists exactly one model structure on the category of small categories for which the weak equivalences are the categorical equivalences. 
So the Joyal-Tierney model structure on $\SmCat$ is an extremely fundamental one: it is the only model structure on $\SmCat$ whose associated homotopy theory is exactly the study of small categories up to weak equivalence. It is a natural question to ask for a characterization of the CW-complexes in this model category.

In this paper we compute the CW-complexes\footnote{Our definition of a CW-complex in a general model category is really only defined {\em up to weak equivalence}; see the comment preceding Definition~\ref{def of spheres} for some explanation of this. Consequently, when comparing our results with the classical results for topological spaces, a more precise statement is that we compute the {\em weak equivalence types} of all CW-complexes in $\Cat\Sh(\Zar(X))$.} in the Joyal-Tierney model category $\Cat\Sh(\Zar(X))$ of internal categories in the Grothendieck topos $\Sh(\Zar(X))$ of sheaves of sets on a connected topological space $X$.
Our results are as follows: 
\begin{itemize}
\item
In the special case of greatest interest, $X = \pt$, the category
$\Cat\Sh(\Zar(X))$ is simply the category $\SmCat$ of small categories.
In Theorem~\ref{1-complexes} we prove that the $1$-dimensional CW-complexes in $\SmCat$ are precisely the groupoids with the property that the automorphism group of each object is a free group.
In Theorem~\ref{2-complexes} we prove that the $2$-dimensional CW-complexes in $\SmCat$ are precisely the groupoids.
In Corollary~\ref{n-complexes} we prove that each CW-complex in $\SmCat$ is weakly equivalent to a $2$-dimensional CW-complex, and in Corollary~\ref{main cor on cw-cplxs 1} we conclude that the category of CW-complexes in $\SmCat$ is precisely the category of groupoids.

In classical homotopy theory, the category of topological spaces homotopy equivalent to a CW-complex was heavily studied (see e.g.~\cite{MR0100267}). As a consequence of our Corollary~\ref{n-complexes}, we know that, if one asks what {\em small categories} have the homotopy type\footnote{``Homotopy type'' and ``weak homotopy type'' are equivalent notions in the Joyal-Tierney model structure on $\SmCat$, since every object in this model category is both cofibrant and fibrant, so every weak equivalence has a homotopy inverse, as in Lemma~4.24 of~\cite{MR1361887}.} of CW-complexes in small categories, the answer is very simple: it is precisely the groupoids.
\item
In Theorem~\ref{connected space corollary} we prove that,
if $X$ is a connected topological space and if we write $\Zar(X)$ for the site of open subsets of $X$,
 then the homotopy category of CW-complexes in $\Cat\Sh(\Zar(X))$ is equivalent to the homotopy category of CW-complexes on $\SmCat$, i.e., the homotopy category of CW-complexes in $\Cat\Sh(\Zar(X))$ is equivalent to the homotopy category of groupoids. 
(More specifically: Theorem~\ref{main thm on cw-cplxs} shows that the CW-complexes in $\Cat\Sh(\Zar(X))$ are, up to weak equivalence, simply the ``constantifications'' of the CW-complexes in internal categories in sheaves on a point, i.e., the CW-complexes in $\SmCat$.)

The categories of CW-complexes we consider in this paper are not themselves model categories, although they embed in the model categories constructed by Joyal and Tierney. Consequently our equivalences of homotopy categories of CW-complexes do not arise from Quillen equivalences. It is worth pointing out that Theorem~\ref{connected space corollary} is {\em not true} if one omits the restriction to CW-complexes; i.e., the categories $\Cat(\Sh(\Zar(X))$ and $\SmCat$ are not equivalent, and their homotopy categories are not equivalent; one only arrives at an equivalence after restricting to the homotopy category of CW-complexes in each category.
\item 
The cofibrant pointed objects $\mathcal{C}^{cof}_*$ in a model category $\mathcal{C}$ possess the structure of a Waldhausen category, and hence have associated algebraic $K$-theory groups $K_n(\mathcal{C}^{cof}_*) = \pi_n\left( \Omega \left| wS_{\bullet}\mathcal{C}^{cof}_*\right|\right)$, as in~\cite{MR802796}. In Theorem~\ref{k-thy vanishing thm 2}, we show that the algebraic $K$-theory groups $K_n(\SmCat_*)$ vanish for all $n$. This vanishing theorem uses homotopical properties of suspension in $\SmCat_*$, and does not use an ``Eilenberg swindle.''
Consequently it holds more generally, for all reasonable sub-Waldhausen-categories of $\SmCat_*$, and in Theorem~\ref{k-thy vanishing thm 2} we prove the following:
suppose $\mathcal{A}$ is a full sub-Waldhausen category of $\SmCat_*$ 
which satisfies the conditions:
\begin{enumerate}
\item pushouts along cofibrations computed in $\mathcal{A}$ agree with the same pushouts computed in $\SmCat_*$, and
\item if $\mathcal{C}$ is an object of $\SmCat_*$ which is contained in $\mathcal{A}$, then the category is also contained in $\mathcal{A}$ which has object set $\ob\mathcal{C}$ and which has exactly one morphism $X\rightarrow Y$ for each ordered pair $X,Y$ in $\ob\mathcal{C}$.
\end{enumerate}
Then the algebraic $K$-theory groups 
$K_n(\mathcal{A})$ vanish for all $n$.
\end{itemize}

In~\cite{MR2393085}, Leinster constructs a notion of the ``Euler characteristic'' of a finite category satisfying certain hypothesis (namely, the finite category must admit a ``weighting'' and a ``coweighting''). Proposition~2.8 in~\cite{MR2393085} (see also~\cite{MR2804514} for the analogous result in the more general setting of Euler characteristics for categories developed in~\cite{MR2739781}) establishes a kind of additivity for Leinster's Euler characteristic. One can ask if there exists an Euler characteristic $\chi$ on pointed finite categories which satisfies the natural ``cofiber-additivity'' given by cofiber sequences in the Joyal-Tierney model structure---i.e., such that, if
\[ \xymatrix{ \mathcal{C} \ar[r] \ar[d] & \mathcal{D} \ar[d] \\ 1 \ar[r] & \mathcal{E} }\]
is a homotopy pushout square in $\SmCat_*$, then $\chi(\mathcal{D}) = \chi(\mathcal{C}) + \chi(\mathcal{E})$. As a corollary of our Theorem~\ref{k-thy vanishing thm 2}, if $\mathcal{A}$ is a full sub-Waldhausen-category of $\SmCat_*$ satisfying the two conditions described above, then {\em the only cofiber-additive Euler characteristic on $\mathcal{A}$ is identically zero}---since any such Euler characteristic must factor through the universal Euler characteristic landing in $K_0(\mathcal{A}) \cong 0$.

\section{CW-complexes in a model category.}

The following is one of the usual definitions of a homotopy pushout; see~\cite{MR1361887}, for example.
\begin{definition}
Let $\mathcal{C}$ be a model category. We say that an object $W$ of $\mathcal{C}$ is {\em the homotopy pushout} of a given diagram
\begin{equation}\label{diag 4308} \xymatrix{ X \ar[r] \ar[d] & Z  \\ Y  }\end{equation}
if diagram~\eqref{diag 4308} fits into a commutative diagram
\[\xymatrix{
 \tilde{X} \ar[r] \ar[d]\ar@/^{3ex}/[rr] & \tilde{Z} \ar[d]\ar@/_1ex/[rr] & X \ar[r]\ar[d] & Z \\
 \tilde{Y} \ar[r] \ar@/^{3ex}/[rr] & W & Y &  
}\]
in which the left-hand square is a pushout square, the curved arrows are acyclic fibrations, 
the maps $\tilde{X} \rightarrow \tilde{Z}$ and $\tilde{X} \rightarrow \tilde{Y}$ are cofibrations,
and $\tilde{X}$ is cofibrant.
\end{definition}
We write ``the'' homotopy pushout of~\eqref{diag 4308}, although it is not a well-defined isomorphism type; however, it is well-defined up to weak equivalence (see section~10 of~\cite{MR1361887}, for example).

\begin{definition}\label{def of one-sided htpy pushout}
Let $\mathcal{C}$ be a model category,
let $f: X \rightarrow Z$ and $g: X \rightarrow Y$ be maps in $\mathcal{C}$. Suppose that $X$ is cofibrant.
We say that an object $W$ of $\mathcal{C}$ is {\em the one-sided homotopy pushout of $f$ along $g$}
if $f$ and $g$ fit into a commutative diagram
\begin{equation}\label{htpy pushout diagram}\xymatrix{
 X \ar[r]_f \ar[d]\ar@/^{3ex}/[rr]^{\id} & Z \ar[d]\ar@/_1ex/[rr]_<<<{\id} & X \ar[r]\ar[d]^g & Z \\
 \tilde{Y} \ar[r] \ar@/^{3ex}/[rr] & W & Y & 
}\end{equation}
in which the left-hand square is a pushout square, the map $\tilde{Y}\rightarrow Y$ is an acyclic fibration, and
the map $X \rightarrow \tilde{Y}$ is a cofibration.
\end{definition}
Note that Definition~\ref{def of one-sided htpy pushout} is asymmetric: the one-sided homotopy pushout $W$ of $f: X \rightarrow Z$ along $g: X \rightarrow Y$ is equipped with a canonical map $Z \rightarrow W$, but not with a canonical map $Y \rightarrow W$. However, the ordinary homotopy pushout does not have this asymmetry, and in Proposition~\ref{one-sided homotopy pushouts} we show that the one-sided homotopy pushout is weakly equivalent to the ordinary homotopy pushout in any left-proper model category. As an immediate consequence, in a left proper model category, the one-sided homotopy pushout is symmetric up to weak equivalence. Proposition~\ref{one-sided homotopy pushouts} is easy and well-known; we do not know where it appears in the literature, but its proof is a pleasant exercise which we leave to the reader.
\begin{prop}\label{one-sided homotopy pushouts}
Let $\mathcal{C}$ be a left proper model category, that is, every pushout of a weak equivalence along a cofibration in $\mathcal{C}$ is a weak equivalence. 
Then every one-sided homotopy pushout in $\mathcal{C}$ is weakly equivalent to the homotopy pushout.
\end{prop}

In Definition~\ref{def of spheres} we present some straightforward generalizations\footnote{We have chosen a definition of ``spheres in a model category'' which is only well-defined up to weak equivalence, not up to isomorphism. Technically this is not exactly a generalization of the classical notion of spheres in topological spaces, since of course the $n$-sphere is well-defined up to homeomorphism. It is possible to give a more ``rigid,'' up-to-isomorphism definition of spheres in a model category, but requires (inductively) making choices of cofibrations from spheres into contractible objects, and the statements of the results later in the paper are slightly more convoluted-looking if we instead use the ``rigid'' definition. So we prefer to use the up-to-weak-equivalence definition of spheres. (The zero-sphere, however, is well-defined up to isomorphism.)} of the usual notions of spheres and CW-complexes to settings much more general than the classical setting of topological spaces:
\begin{definition}\label{def of spheres}
Let $\mathcal{C}$ be a model category whose terminal object $1$ is cofibrant. 
\begin{itemize}
\item By the  {\em $0$-sphere in $\mathcal{C}$}, written $S^0$, we mean the coproduct $1 \coprod 1$ of two copies of the terminal object $1$ of $\mathcal{C}$. 
\item
If $n$ is a positive integer, then 
an object $X$ of $\mathcal{C}$ is called an {\em $n$-sphere in $\mathcal{C}$} if $X$ can be written as a homotopy pushout
of a diagram of the form
\[\xymatrix{ Y \ar[r]\ar[d] & 1  \\ 1 & }\]
where $Y$ is an $(n-1)$-sphere in $\mathcal{C}$.
We will sometimes write $S^n$ for the weak equivalence class of the $n$-spheres inside the category of cofibrant objects of $\mathcal{C}$.
\item 
An object $X$ of $\mathcal{C}$ is called a {\em $0$-dimensional CW-complex in $\mathcal{C}$} if $X$ is cofibrant and a coproduct of copies of the terminal object $1$. 
\item 
If $n$ is a positive integer, then 
an object $X$ of $\mathcal{C}$ is called an {\em $n$-dimensional CW-complex in $\mathcal{C}$} if $X$ can be written as a one-sided homotopy pushout of a map $\coprod_{i\in I} S_i \rightarrow Y$ along the map $\coprod_{i\in I} p_{S_i}: \coprod_{i\in I} S_i \rightarrow \coprod_{i\in I} 1$,
where $Y$ is an $(n-1)$-dimensional CW-complex in $\mathcal{C}$, 
 $I$ is a set,
$S_i$ is an $(n-1)$-sphere in $\mathcal{C}$ for each $i\in I$, and the map $p_{S_i}: S_i \rightarrow 1$ is the unique map to the terminal object.
The resulting map $ Y\rightarrow X$ is called an {\em $(n-1)$-skeleton inclusion.} 
\item 
An object $X$ of $\mathcal{C}$ is called a {\em finite-dimensional CW-complex in $\mathcal{C}$} if $X$ is an $n$-dimensional CW-complex for some $n$.
\end{itemize}
\end{definition}

\section{Joyal-Tierney model structures.}\label{section on j-t model structure}

\begin{convention}\label{convention on tilde}
Given categories $\mathcal{C},\mathcal{D}$ with pullbacks, and given 
a pullback-preserving functor $F: \mathcal{C} \rightarrow \mathcal{D}$,
we can apply $F$ levelwise to an internal category
\begin{equation}\label{int cat 1} \xymatrix{ \mathcal{O} \ar[r] & \mathcal{M}\ar@<1ex>[l] \ar@<-1ex>[l] &
 \mathcal{M}\times_{\mathcal{O}}\mathcal{M} \ar[l] }\end{equation}
in $\mathcal{C}$, yielding an internal category
\begin{equation}\label{int cat 2} \xymatrix{ F(\mathcal{O}) \ar[r] & F(\mathcal{M})\ar@<1ex>[l] \ar@<-1ex>[l] &
 F(\mathcal{M})\times_{F(\mathcal{O})}F(\mathcal{M}) \ar[l] }\end{equation}
in $\mathcal{D}$.
We adopt the convention that we will write $\tilde{F}$ for the evident functor $\tilde{F}: \Cat(\mathcal{C}) \rightarrow \Cat(\mathcal{D})$ which, on objects, sends~\eqref{int cat 1} to~\eqref{int cat 2}. 

So, for example, when we apply the global sections functor $\Gamma: \Sh(\mathcal{C},\tau) \rightarrow \Sets$ levelwise to an internal category, we are applying the functor $\tilde{\Gamma}: \Cat\Sh(\mathcal{C},\tau) \rightarrow \Cat\Sets = \SmCat$. Similarly, in Proposition~\ref{sheafification} and Definition~\ref{def of c} we introduce functors $\tilde{\#}$ and $\tilde{c}$ which are given by applying sheafification $\#$ and constantification $c$, respectively, levelwise to an internal category. 
\end{convention}

Let $\triv$ denote the trivial site, i.e., the site of open subsets of the one-point topological space.
Since $\SmCat = \Cat\Sh(\triv)$, any statements we make for $\Cat\Sh(\mathcal{C},\tau)$ have corresponding statements for $\SmCat$ as a special case.

Recall that an {\em internal category} in a category $\mathcal{A}$ with finite limits is
a small category object in $\mathcal{A}$, i.e., a pair of objects $(\mathcal{O},{\mathcal{M}})$ in $\mathcal{A}$ and structure maps
\begin{eqnarray*} 
 \eta_L: {\mathcal{M}} & \rightarrow & \mathcal{O} \\
 \eta_R: {\mathcal{M}} & \rightarrow & \mathcal{O} \\
 \epsilon: \mathcal{O} & \rightarrow & {\mathcal{M}} \\
 \Delta: {\mathcal{M}} \times_{\mathcal{O}} {\mathcal{M}} & \rightarrow & {\mathcal{M}} \end{eqnarray*}
encoding domain of a morphism, codomain of a morphism, identity morphism on an object, and composition of a composable pair of morphisms, respectively, 
and making the inevitable diagrams commute which encode associativity and unitality of composition.
The pullback ${\mathcal{M}}\times_{\mathcal{O}} {\mathcal{M}}$ is taken using the two maps $\eta_L,\eta_R$ as the maps ${\mathcal{M}}\rightarrow {\mathcal{O}}$ in the defining pullback square.  We write $j_L$ for the projection of a composable pair to its left morphism and $j_R$ for the projection to its right morphism:
\begin{center}
 \begin{tikzcd}
  {\mathcal{M}}\times_{\mathcal{O}} {\mathcal{M}} \arrow[swap]{d}{j_L} \arrow{r}{j_R} &\mathcal{M} \arrow{d}{\eta_L}\\
  \mathcal{M} \arrow[swap]{r}{\eta_R} &\mathcal{O}
 \end{tikzcd}
\end{center}

Given an internal category $({\mathcal{O}},{\mathcal{M}})$, we have (see e.g. \cite{MR558105}) its coreflection $({\mathcal{O}},\Iso({\mathcal{O}},{\mathcal{M}}))$ in internal groupoids, and the canonical map
\begin{equation}\label{coreflector 132} ({\mathcal{O}}, {\mathcal{M}})\leftarrow ({\mathcal{O}}, \Iso({\mathcal{O}}, {\mathcal{M}})).\end{equation}
The following definitions are special cases of those in \cite{MR558105}; see also \cite{MR2210576}.
\begin{definition}\label{model structure}
Suppose $\mathcal{A}$ is a category with finite limits and $({\mathcal{O}}_0,{\mathcal{M}}_0)$ and $({\mathcal{O}}_1,{\mathcal{M}}_1)$ are internal categories in $\mathcal{A}$.
We say that a map $(f_{\mathcal{O}},f_{{\mathcal{M}}}): ({\mathcal{O}}_0,{\mathcal{M}}_0)\rightarrow ({\mathcal{O}}_1,{\mathcal{M}}_1)$ of internal categories is:
\begin{itemize}
\item {\em full and faithful} if the commutative square in $\mathcal{A}$
\[\xymatrix{
{\mathcal{M}}_0 \ar[d]^{\left( \eta_L,\eta_R\right)} \ar[r]^{f_{{\mathcal{M}}}} & {\mathcal{M}}_1 \ar[d]^{\left( \eta_L,\eta_R\right)} \\
{\mathcal{O}}_0\times {\mathcal{O}}_0 \ar[r]^{f_{\mathcal{O}} \times f_{\mathcal{O}}} & {\mathcal{O}}_1\times {\mathcal{O}}_1 }\]
is a pullback square,
\item {\em essentially surjective} if the composite map $\eta_R\circ\iota \circ p$ in the diagram
\[\xymatrix{
 {\mathcal{O}}_0\times_{{\mathcal{O}}_1} \Iso({\mathcal{O}}_1,{\mathcal{M}}_1) \ar[r]^{p} \ar[d] & \Iso({\mathcal{O}}_1,{\mathcal{M}}_1) \ar[r]^{\eta_R\circ\iota}\ar[d]^{\eta_L\circ\iota} & {\mathcal{O}}_1 \\
 {\mathcal{O}}_0 \ar[r]^{f_{\mathcal{O}}} & {\mathcal{O}}_1 & }\]
is a regular epimorphism in $\mathcal{A}$ where $\iota: \Iso({\mathcal{O}}_1,{\mathcal{M}}_1)\to{\mathcal{M}}_1$ is the insertion of isomorphisms into all morphisms, and
\item {\em an internal categorical equivalence} if it is full and faithful and essentially surjective.
\end{itemize}
\end{definition}

\begin{remark}\label{iota-def}
The map $\iota$ can be defined categorically when working in a category $\mathcal{E}$ with pullbacks.  Given an internal category $({\mathcal{O}},{\mathcal{M}})$ in $\mathcal{E}$, with structure maps $\eta_L$, $\eta_R$, $\epsilon$, and $\Delta$, and pullback projections $j_L$ and $j_R$ as described earlier, the pullback of $\epsilon$ along $\Delta$ gives a subobject $P({\mathcal{O}},{\mathcal{M}})$ of ${\mathcal{M}}\times_{\mathcal{O}}{\mathcal{M}}$.
\begin{center}
 \begin{tikzcd}
  P(\mathcal{O},\mathcal{M}) \arrow{d} \arrow{r}{\rho} &{\mathcal{M}}\times_{\mathcal{O}} {\mathcal{M}} \arrow{d}{\Delta}\\
  \mathcal{O} \arrow[swap]{r}{\epsilon} &\mathcal{M}
 \end{tikzcd}
\end{center}
(In sets $P({\mathcal{O}},{\mathcal{M}})$ is the pairs of morphisms which compose to an identity morphism.)  Let $\rho: P({\mathcal{O}},{\mathcal{M}})\to{\mathcal{M}}\times_{\mathcal{O}}{\mathcal{M}}$ be the projection of this pullback, and now take the following pullback.
\begin{center}
 \begin{tikzcd}
  Q(\mathcal{O},\mathcal{M}) \arrow{d} \arrow{r}{\rho'} &P(\mathcal{O},\mathcal{M}) \arrow{d}{j_L\circ\rho}\\
  P(\mathcal{O},\mathcal{M}) \arrow[swap]{r}{j_R\circ\rho} &\mathcal{M}
 \end{tikzcd}
\end{center}
(In sets $Q({\mathcal{O}},{\mathcal{M}})$ is the collection of $4$-tuples $(a,b,c,d)$ of morphisms for which the first pair $b\circ a = 1$ composes to an identity morphism, the last pair $d\circ c$ composes to an identity morphism, and the middle pair $b = c$ are equal.)  This object $Q(\mathcal{O},\mathcal{M})$ may be identified with $Iso({\mathcal{O}},{\mathcal{M}})$, and the diagonal composite $j_L\circ\rho\circ\rho'$ of this square (which in sets sends the tuple to the middle morphism $c$), may be identified with the map $\iota$.
\end{remark}

The following model structure was constructed in \cite{MR1097495}.
\begin{theorem}\label{j-t thm} {\bf (Joyal-Tierney.)}
Suppose $(\mathcal{C},\tau)$ is a site. We denote the Grothendieck topos of sheaves of sets on $(\mathcal{C},\tau)$ as $\Sh(\mathcal{C},\tau)$. There exists
a cofibrantly generated model structure on $\Cat\Sh(\mathcal{C},\tau)$, called the
{\em Joyal-Tierney model structure}, defined as follows:
\begin{itemize}
\item The cofibrations are the maps of internal categories $({\mathcal{O}}_0,{\mathcal{M}}_0)\rightarrow ({\mathcal{O}}_1,{\mathcal{M}}_1)$ 
whose underlying map ${\mathcal{O}}_0\rightarrow {\mathcal{O}}_1$ is a monomorphism in the category $\Sh(\mathcal{C},\tau)$.
\item The weak equivalences are the internal categorical equivalences.
\item The fibrations are the maps with the right lifting property with respect to the acyclic cofibrations.
\end{itemize}
\end{theorem}

\begin{observation}\label{j-t is left-proper}
Recall ``Theorem B'' from C. Reedy's well-known thesis~\cite{reedy}: given a cofibration $i: A \rightarrow B$ and a weak equivalence $f: A \rightarrow C$, if $A$ is cofibrant, then the pushout map $B \rightarrow C\coprod_A B$ is also a weak equivalence. An easy consequence of Reedy's Theorem B is that a model category is left-proper if all its objects are cofibrant. (See the statement of Proposition~\ref{one-sided homotopy pushouts} for the definition of left-properness.)
Consequently the Joyal-Tierney model structure on $\Cat\Sh(\mathcal{C},\tau)$ is left-proper.
\end{observation}

\begin{remark}\label{j-t special case}
Let $(\mathcal{C},\tau)$ be the site of open subsets of a one-point space.
The resulting special case of Theorem~\ref{j-t thm} is of particular interest: Theorem~\ref{j-t thm} gives a model structure on $\SmCat$ such that
\begin{itemize}
\item its cofibrations are the functors which are injective on objects, 
\item its weak equivalences are the functors which induce an equivalence of categories,
\item and its fibrations are the isofibrations, i.e., the functors $F: \mathcal{X} \rightarrow \mathcal{Y}$ such that, for each object $X$ of $\mathcal{X}$ and each isomorphism $f: F(X) \stackrel{\cong}{\longrightarrow} Y$ in $\mathcal{Y}$, there exists an isomorphism $\tilde{f}: X \stackrel{\cong}{\longrightarrow} \tilde{Y}$ in $\mathcal{X}$ such that $F(\tilde{f}) = f$.\end{itemize}
\end{remark}

The following result is classical; see e.g. Theorem~III.5.1 in~\cite{MR1300636}.
\begin{prop}{\bf Sheafification.} 
\label{sheafification} Let $(\mathcal{C},\tau)$ be a site, with $\mathcal{C}$ small, and let $\mathcal{D}$
be a category with equalizers and small colimits. 
Write $\mathcal{D}\Sh(\mathcal{C},\tau)$ for the category of $\mathcal{D}$-valued sheaves on $(\mathcal{C},\tau)$.
The forgetful functor
\[ \forget: \mathcal{D}\Sh(\mathcal{C},\tau)\stackrel{\forget}{\longrightarrow} 
\mathcal{D}^{\mathcal{C}^{\op}}\]
has a left adjoint $\#$, which we call ``sheafification.'' The functor $\#$ commutes with finite limits.
\end{prop}

\begin{lemma}\label{internal cat sheafification is a left adjoint}
Suppose $\mathcal{C}$ is small and $(\mathcal{C},\tau)$ is a cofiltered site.
The forgetful functor \[ \tilde{\forget}: \Cat\Sh(\mathcal{C},\tau) \rightarrow \Cat\left( \Sets^{\mathcal{C}^{\op}}\right),\]
from internal categories in $\Sh(\mathcal{C},\tau)$ to internal categories in $\Sets^{\mathcal{C}^{\op}}$,
has a left adjoint $\tilde{\#}: \Cat\left( \Sets^{\mathcal{C}^{\op}}\right) \rightarrow \Cat\Sh(\mathcal{C},\tau)$
which we also call sheafification.
This sheafification $\tilde{\#}$ commutes with forgetting to the underlying object and morphism (pre)sheaves, in the sense that
the underlying object sheaf of $\tilde{\#}(\mathcal{X})$ is naturally isomorphic to the sheafification of
the underlying object presheaf of $\mathcal{X}$, and the
underlying morphism sheaf of $\tilde{\#}(\mathcal{X})$ is naturally isomorphic to the sheafification of the underlying morphism presheaf of $\mathcal{X}$.
\end{lemma}
\begin{proof}
From Proposition~\ref{sheafification} and the stated assumptions 
one knows that $\#$ commutes with finite limits,
so if 
\[ \xymatrix{ \mathcal{O} \ar[r] & \mathcal{M}\ar@<1ex>[l] \ar@<-1ex>[l] &
 \mathcal{M}\times_{\mathcal{O}}\mathcal{M} \ar[l] }\]
is an internal category in $\Sets^{\mathcal{C}^{\op}}$, applying
$\#$ yields 
\[ \xymatrix{ \#\mathcal{O} \ar[r] & \#\mathcal{M}\ar@<1ex>[l] \ar@<-1ex>[l] &
 \#\mathcal{M}\times_{\#\mathcal{O}}\#\mathcal{M}\ar[l], }\]
an internal category in $\Sh(\mathcal{C},\tau)$.

Now suppose $({\mathcal{O}},{\mathcal{M}})$ is an internal category in 
$\Sh(\mathcal{C},\tau)$ and $f: (\mathcal{O}^{\prime},\mathcal{M}^{\prime})\rightarrow ({\mathcal{O}},{\mathcal{M}})$
a morphism of internal categories. The universal property of $\#$
and its naturality then produces a unique factorization of $f$
through the canonical map 
$(\mathcal{O}^{\prime},\mathcal{M}^{\prime})\rightarrow (\#\mathcal{O}^{\prime},\#\mathcal{M}^{\prime})$
in the category of internal categories in $\Sets^{\mathcal{C}^{\op}}$.
Hence $(\mathcal{O}^{\prime},\mathcal{M}^{\prime})\mapsto (\#\mathcal{O}^{\prime},\#\mathcal{M}^{\prime})$
is a reflection functor 
$\Cat(\Sets^{\mathcal{C}^{\op}}) \rightarrow \Cat\Sh(\mathcal{C},\tau)$,
i.e., it is left adjoint to the inclusion 
$\Cat\Sh(\mathcal{C},\tau)\rightarrow \Cat(\Sets^{\mathcal{C}^{\op}})$.
\end{proof}

\begin{lemma}\label{constantification is a left adjoint}
Suppose $\mathcal{C}$ has a terminal object.
Let $\tilde{c}$ be the 
functor \[ \tilde{c}: \SmCat \rightarrow \Cat\left( \Sets^{\mathcal{C}^{\op}}\right),\]
which sends a small category $\mathcal{A}$ to
the constant presheaf taking the value $\mathcal{A}$.
Then $\tilde{c}$ is a left adjoint.
\end{lemma}
\begin{proof}
For each object $U$ of $\mathcal{C}$, let $\tilde{\Gamma}_U: \Cat\left(\Sets^{\mathcal{C}^{\op}}\right)\rightarrow\SmCat$ be the 
functor ``evaluation at $U$,'' i.e., $\tilde{\Gamma}_U\left( \mathcal{M},\mathcal{O}\right)$ is the small category
with objects $\mathcal{M}(U)$ and morphisms $\mathcal{O}(U)$.

The functor $\tilde{\Gamma}_U$ has a left adjoint $\tilde{c}_U: \SmCat\rightarrow\Cat\left(\Sets^{\mathcal{C}^{\op}}\right)$
given by sending a small category $\mathcal{A}$ to the internal category $\mathcal{X}$ in $\Sets^{\mathcal{C}^{\op}}$
defined as follows: for all $V$ in $\mathcal{C}$, we let $\mathcal{X}(V) = \mathcal{A}$ if there exists a morphism
$V\rightarrow U$ in $\mathcal{C}$, and we let $\mathcal{X}(V) = \emptyset$ if there does not exist a morphism
$V\rightarrow U$ in $\mathcal{C}$. Verifying that $\tilde{c}_U$ is left adjoint to $\tilde{\Gamma}_U$ is elementary.

The functor $\tilde{c}$ in the statement of the lemma is simply $\tilde{c}_1$, with $1$ the terminal object in $\mathcal{C}$.
\end{proof}

\begin{definition}\label{def of c} Define $c: \Sets \to \Sets^{\mathcal{C}^{\op}}$ to be the functor taking a set $X$ to the constant presheaf with value $X$. 
\end{definition} 
It is clear that $\tilde{c}$ is given by applying $c$ to the diagrams.
\[ \xymatrix{ \mathcal{O} \ar[r] & \mathcal{M}\ar@<1ex>[l] \ar@<-1ex>[l] &
 \mathcal{M}\times_{\mathcal{O}}\mathcal{M} \ar[l] }\]
Moreover, because limits and colimits in presheaves are computed levelwise, both are preserved by $c$.

\begin{theorem}\label{spheres to spheres}
Suppose $\mathcal{C}$ is small and has a terminal object. Then the following statements hold:
\begin{itemize}
\item The composite $\tilde{\#}\circ\tilde{c}: \SmCat \rightarrow \Cat\Sh(\mathcal{C},\tau)$ is left adjoint to global sections $\tilde{\Gamma}$.
\item The composite $\tilde{\#}\circ\tilde{c}$ preserves finite limits.
\item The composite $\tilde{\#}\circ\tilde{c}$ sends cofibrations to cofibrations.
\item The composite $\tilde{\#}\circ\tilde{c}$ sends cofibrant objects to cofibrant objects.
\item The composite $\tilde{\#}\circ\tilde{c}$ sends weak equivalences to weak equivalences.
\item The composite $\tilde{\#}\circ\tilde{c}$ sends $n$-spheres to $n$-spheres. (More precisely: if $X$ is an $n$-sphere in $\SmCat$, then $(\tilde{\#}\circ\tilde{c})(X)$ is an $n$-sphere in $\Cat\Sh(\mathcal{C},\tau)$.) 
\end{itemize}
\end{theorem}
\begin{proof}
\begin{itemize}

\item In Lemma~\ref{constantification is a left adjoint} and Lemma~\ref{internal cat sheafification is a left adjoint} we prove that
$\tilde{c}$ and $\tilde{\#}$ are left adjoints, so the same is true of their composite.  
\item The functor {$\tilde{c}$ preserves products simply because products in functor categories are computed pointwise. The functor $\tilde{\#}$ preserves finite limits by its definition: an internal category
\[ \xymatrix{ \mathcal{O} \ar[r] & \mathcal{M}\ar@<1ex>[l] \ar@<-1ex>[l] &
 \mathcal{M}\times_{\mathcal{O}}\mathcal{M} \ar[l] }\]
sheafifies as
\[ \xymatrix{ \#\mathcal{O} \ar[r] & \#\mathcal{M}\ar@<1ex>[l] \ar@<-1ex>[l] &
 \#\mathcal{M}\times_{\#\mathcal{O}}\#\mathcal{M}\ar[l], }\],
 and since $\#$ preserves finite limits, $\tilde{\#}$ does as well. }
\item If $(f_{\mathcal{O}},f_{\mathcal{M}}): (\mathcal{O},\mathcal{M})\rightarrow ({\mathcal{O}^{\prime}},{\mathcal{M}^{\prime}})$ is a cofibration in $\SmCat$, 
then $f_{\mathcal{O}}$ is a monomorphism in $\Sets$, so $(\tilde{c}(f))_{\mathcal{O}}$ is a monomorphism in $\Sets^{\mathcal{C}^{\op}}$ by construction,
and $\left(\left(\tilde{\#}\circ \tilde{c}\right)(f)\right)_{\mathcal{O}} \cong \#(\tilde{c}(f)_{\mathcal{O}})$.
The functor $\#$ preserves finite limits by Proposition~\ref{sheafification}, so in particular it preserves pullbacks, hence preserves the kernel pair of any morphism. It is a standard exercise (in fact, see exercise III.4.4 in~\cite{MR1712872} for the dual) that a map $g: X \rightarrow Y$ is monic if and only if the pullback of the kernel pair of $g$ is $X$; consequently $\#$ preserves monomorphisms.
So $\left(\left(\tilde{\#}\circ \tilde{c}\right)(f)\right)_{\mathcal{O}}$ is a monomorphism.
So $\left(\tilde{\#}\circ \tilde{c}\right)(f)$ is a cofibration.
\item Since $\tilde{\#}\circ \tilde{c}$ preserves finite limits, it sends the intial object to an initial object, so $\tilde{\#}\circ \tilde{c}$ preserving cofibrations immediately implies that it also sends cofibrant objects to cofibrant objects.
\item {Let $F$ be a weak equivalence, so that it is full and faithful as well as essentially surjective.  Because $\#$ and $c$ both preserve finite limits, pullbacks in particular are preserved and $\left(\tilde{\#}\circ\tilde{c}\right)(F)$ remains full and faithful.  Since $\#$ and $c$ each preserve colimits, they each send regular epimorphisms to regular epimorphisms, so essential surjectivity of $\eta_R\circ\iota\circ p$ will be preserved as long as the map $\eta_L\circ\iota$ is preserved by $\tilde{\#}$ and by $\tilde{c}$.  This follows from the construction of $\iota$ as a finite limit in \ref{iota-def}, and the fact that $\tilde{\#}\circ \tilde{c}$ preserves finite limits.}
\item 
Since $\tilde{\#}\circ\tilde{c}$ preserves products and colimits, in particular it preserves the coproduct of two copies of the trivial product, i.e.,
$1\coprod 1$. 
This is the initial step in an induction.

Since $\tilde{\#}\circ\tilde{c}$ preserves colimits, cofibrations, cofibrant objects, and weak equivalences, it preserves homotopy colimits (up to weak equivalence), and in particular,
it preserves homotopy pushouts (up to weak equivalence). 
This provides the inductive step: if $\tilde{\#}\circ \tilde{c}$ sends $n$-spheres to $n$-spheres, then it sends $n+1$-spheres to $n+1$-spheres.
So, for all nonnegative integers $n$, we have that $\tilde{\#}\circ \tilde{c}$ sends $n$-spheres to $n$-spheres.
\end{itemize}
\end{proof}

\begin{corollary}
 The functor $\tilde{\Gamma}$ sends weak equivalences between fibrant objects to weak equivalences.
\end{corollary}
\begin{proof}
 Since $\tilde{\Gamma}$ is right adjoint to the left Quillen functor $\tilde{\#}\circ \tilde{c}$, it preserves fibrations and acyclic fibrations. So by Ken Brown's lemma (see e.g. Lemma~1.1.12 of~\cite{MR1650134}), $\tilde{\Gamma}$ preserves weak equivalences between fibrant objects.
\end{proof}

\section{Exotic attaching maps}
Since we have all spheres in the image of $\tilde{\#}\circ\tilde{c}$, it is natural to ask whether all CW-complexes are also in the image. 
Under some assumptions on the underlying site, this is indeed the case, up to weak equivalence; see Theorem~\ref{connected space}. 
 The following example shows, however, that some assumptions on the underlying site are necessary: there exist sites for which not all attaching maps between spheres can be given by application of $\tilde{\#}\circ\tilde{c}$.  We think of maps outside the image of the functor $\tilde{\#}\circ\tilde{c}$ as ``exotic attaching maps,'' since one doesn't expect the homotopy cofiber of such maps to be in the essential image of $\tilde{\#}\circ \tilde{c}$.

\begin{example}
 By Theorem~\ref{spheres to spheres}, $S^0$ in $\Cat\Sh(\mathcal{C},\tau)$ is $\left(\tilde{\#}\circ\tilde{c}\right)(1\amalg 1)$, that is, the sheafification of the constant presheaf which evaluates to the two-element discrete category. 
Since the sheafification functor $\tilde{\#}: \Cat(\Sets^{\mathcal{C}}) \rightarrow \Cat\Sh(\mathcal{C},\tau)$ is given by applying the sheafification functor $\#: \Sets^{\mathcal{C}} \rightarrow \Sh(\mathcal{C},\tau)$ levelwise, the internal category $\left(\tilde{\#}\circ\tilde{c}\right)(1\amalg 1)$ is discrete (i.e., its morphisms $\eta_L,\eta_R,$ and $\epsilon$ are isomorphisms).
Regarded as a sheaf of small categories, $\left(\tilde{\#}\circ\tilde{c}\right)(1\amalg 1)$
is simply the constant sheaf taking value the two-element set $\mathbf{2}$.
In the case that our site is a space $X$, the standard identification of constant sheaves (i.e., sheafifications of constant presheaves) on topological spaces, as in exercise 2.7 of \cite{MR1300636}, then gives us that 
$\left(\tilde{\#}\circ\tilde{c}\right)(1\amalg 1)$ is the sheaf of locally constant functions from $X$ to $\mathbf{2}$.  Let $X$ be the discrete two-element space, with elements $u$ and $v$,
so that 
\[ \left(\left(\tilde{\#}\circ\tilde{c}\right)(1\amalg 1)\right)\left( \{ u\}\right)  = 
\left(\left(\tilde{\#}\circ\tilde{c}\right)(1\amalg 1)\right)\left( \{ v\}\right)  = \mathbf{2}.\]
Fix an element $0\in \mathbf{2}$.
To construct a map 
\[ \xi: \left(\tilde{\#}\circ\tilde{c}\right)(1\amalg 1)\rightarrow \left(\tilde{\#}\circ\tilde{c}\right)(1\amalg 1)\]
which is not in the image of $\tilde{\#}\circ\tilde{c}$, we define $\xi(\{u\}): \mathbf{2} \rightarrow\mathbf{2}$ to be the identity map on $\mathbf{2}$, and we define $\xi(\{v\}): \mathbf{2} \rightarrow\mathbf{2}$ to be the map sending both elements of $\mathbf{2}$ to $0\in \mathbf{2}$.
As $\{u\}$ and $\{v\}$ have empty intersection, this gives a unique map on their union.  Since $\{u\}$ and $\{v\}$ form an open cover of $X$, the definitions of $\xi(\{u\})$ and $\xi(\{v\})$ completely determine how $\xi$ must act on the global sections.  Furthermore, the induced stalk map $\xi_{\{u\}}$ is not isomorphic to the induced stalk map $\xi_{\{v\}}$, so $\xi$ is an example of an ``exotic attaching map.'' 
\end{example}

In the case where our site is a connected space, however, $\tilde{\#}\circ\tilde{c}$ is a full functor. This follows from the following proposition:
\begin{prop}\label{unit isomorphism}
Let $\mathcal{C}$ be the site of open subsets of a connected topological space. Then the unit natural transformation 
\[ \tilde{\eta} : \id_{\Cat(\Sets^{\mathcal{C}})} \rightarrow \tilde{\Gamma}\circ\tilde{\#}\circ \tilde{c} \]
of the adjunction $\tilde{\#}\circ\tilde{c}\dashv\tilde{\Gamma}$ is a natural isomorphism.

Consequently, $\tilde{\#}\circ\tilde{c}$ is a full and faithful functor.
\end{prop}
\begin{proof}
By standard facts on sheafification of constant presheaves (see e.g. exercise 2.7 of~\cite{MR1300636}), the sheafification of the constant presheaf of sets taking values in a set $Z$ has $Z$ as its global sections. Consequently the unit natural transformation 
\[ \eta: \id_{\Sets^{\mathcal{C}}} \to \Gamma\circ\#\circ c\] 
is a natural isomorphism. (Recall Convention~\ref{convention on tilde}: the symbols $\Gamma,\#,$ and $c$ refer to the classical constructions on sheaves of sets, while their levelwise applications to internal categories in sheaves of sets are decorated with tildes: $\tilde{\Gamma},\tilde{\#},$ and $\tilde{c}$.)
If $\mathcal{X}$ is a category object
\[ \xymatrix{ \mathcal{O} \ar[r] & \mathcal{M}\ar@<1ex>[l] \ar@<-1ex>[l] &
 \mathcal{M}\times_{\mathcal{O}}\mathcal{M} \ar[l] }\]
in sets, then application of $\eta$ necessarily yields an isomorphism on $\mathcal{O}, \mathcal{M},$ and $\mathcal{M}\times_{\mathcal{O}}\mathcal{M}$, while naturality gives the compatibility of the structure maps of $\mathcal{X}$ with those of $(\tilde{\Gamma}\circ\tilde{\#}\circ\tilde{c})(\mathcal{X})$. Consequently $\tilde{\eta}$ is an isomorphism.

It is a general (and elementary, but we still provide the proof) fact that, given functors $F: \mathcal{A} \rightarrow \mathcal{B}$ and $G: \mathcal{B} \rightarrow \mathcal{A}$ and an adjunction $F\dashv G$ whose unit natural transformation $\eta: \id_{\mathcal{A}} \rightarrow GF$ is invertible, the functor $F$ is full and faithful. This is simply because the composite map
\begin{equation}\label{composite 2308} \hom_{\mathcal{A}}(X,Y) \rightarrow \hom_{\mathcal{B}}(FX,FY) \stackrel{\cong}{\longrightarrow} \hom_{\mathcal{A}}(X,GFY)\end{equation}
is postcomposition with the unit map $Y \rightarrow GFY$, which is a bijection, and so the left-hand map in~\eqref{composite 2308} is also a bijection.
In the case where $F = \tilde{\#}\circ\tilde{c}$ and $G = \tilde{\Gamma}$, we get that $\tilde{\#}\circ\tilde{c}$ is full and faithful.
\end{proof}

\begin{lemma}
 For every nonempty small category $\mathcal{A}'$, there exists a contractible small category $\mathcal{T}$ and a cofibration $i: \mathcal{A}'\to\mathcal{T}$ such that the image of $i$ under $\tilde{\#}\circ\tilde{c}$ remains a cofibration into a contractible object.
\end{lemma}
\begin{proof}
 By Remark~\ref{j-t special case}, cofibrations in small categories under the Joyal-Tierney model structure are simply functors which are injective on objects, and weak equivalences are honest equivalences of categories.  For a nonempty set $S$, let $\mathbf{C}_S$ be the category with objects $S$ and exactly one arrow in every hom-set.  Such a category is necessarily equivalent to the terminal category $1$.  Writing $\lvert\mathcal{A}'\rvert$ for the set of objects of $\mathcal{A}'$, there is a unique map $i$ from $\mathcal{A}'$ to $\mathcal{T} := \mathbf{C}_{\lvert\mathcal{A}'\rvert}$ which is the identity on objects; $i$ is necessarily a cofibration.  Its image under $\tilde{\#}\circ\tilde{c}$ remains a cofibration into a contractible object by Theorem~\ref{spheres to spheres}.
\end{proof}

\begin{theorem}\label{connected space}
Let $X$ be a connected topological space, and let $Z$ be a $n$-dimensional CW-complex in $\Cat\Sh(X)$.
Then there exists a weak equivalence 
\[ Z \rightarrow \left(\tilde{\#}\circ\tilde{c}\right)(\overline{Z})\]
in $\Cat\Sh(X)$, where $\overline{Z}$ is an $n$-dimensional CW-complex in $\SmCat$.
\end{theorem}
\begin{proof}
Our proof is by induction on $n$. We begin with the base case, $n=0$: a zero-dimensional CW-complex in $\Cat\Sh(X)$ is, by Definition~\ref{def of spheres}, a cofibrant object of $\Cat\Sh(X)$ which is a coproduct of terminal objects. By Theorem~\ref{spheres to spheres}, $\tilde{\#}\circ\tilde{c}$ preserves finite limits, so in particular it sends the terminal object to the terminal object; since $\tilde{\#}\circ\tilde{c}$ is also a left adjoint, it preserves coproducts, so $\tilde{\#}\circ \tilde{c}$ applied to a coproduct of terminal objects is a coproduct of terminal objects. So each zero-dimensional CW-complex in $\Cat\Sh(X)$ is $\tilde{\#}\circ\tilde{c}$ applied to a zero-dimensional CW-complex in $\SmCat$.

Now for the inductive step: choose an $n$-dimensional CW-complex $\mathcal{A}$ in $\Cat\Sh(X)$.  The inductive hypothesis is that there exists a weak equivalence $\varphi: \mathcal{A}\to\tilde{\#}\circ\tilde{c}(\mathcal{A}')$,  where $\mathcal{A}'$ is an $n$-dimensional CW-complex in $\SmCat$.  We want to show that the homotopy pushout of the diagram
 \begin{center}
  \begin{tikzcd}
   \coprod S^n \arrow{d} \arrow{r}{\psi} &\mathcal{A}\\
   \coprod 1
  \end{tikzcd}
 \end{center}
 is weakly equivalent to the image of an $(n + 1)$-complex under $\tilde{\#}\circ\tilde{c}$, for any coproduct of $n$-spheres and any attaching map $\psi$.
 
 By Theorem~\ref{spheres to spheres}, there exists a weak equivalence $\theta: \left(\tilde{\#}\circ\tilde{c}\right)\left(S^n\right)\to S^n$, and the previous lemma gives a cofibration $i: S^n\to\mathcal{T}$ in small categories, with $\mathcal{T}$ contractible, and such that its image under $\tilde{\#}\circ\tilde{c}$ remains a cofibration with contractible codomain.  Taking the pushout
 \begin{center}
  \begin{tikzcd}
   \left(\tilde{\#}\circ\tilde{c}\right)\left(S^n\right) \arrow[swap]{d}{\left(\tilde{\#}\circ\tilde{c}\right)(i)} \arrow{r}{\theta} &S^n \arrow{d}{i'}\\
   \left(\tilde{\#}\circ\tilde{c}\right)\left(\mathcal{T}\right) \arrow{r} &\mathcal{T}'
  \end{tikzcd}
 \end{center}
 yields $\mathcal{T}'$ weakly equivalent to the contractible object $\tilde{\#}\circ\tilde{c}(\mathcal{T})$ by left-properness (see Observation~\ref{j-t is left-proper}), and a cofibration $i'$.  Now $\tilde{\#}\circ\tilde{c}$ is a left adjoint so preserves colimits, a coproduct of cofibrations remains a cofibration, and coproducts commute with pushouts.  Altogether, the previous pushout diagram yields the next pushout diagram:
 \begin{center}
  \begin{tikzcd}
   \left(\tilde{\#}\circ\tilde{c}\right)\left(\coprod S^n\right) \arrow[swap]{d}{\left(\tilde{\#}\circ\tilde{c}\right)(\coprod i)} \arrow{r}{\coprod\theta} &\coprod S^n \arrow{d}{\coprod i'}\\
   \left(\tilde{\#}\circ\tilde{c}\right)\left(\coprod\mathcal{T}\right) \arrow{r} &\coprod\mathcal{T}',
  \end{tikzcd}
 \end{center}
 where both $\left(\tilde{\#}\circ\tilde{c}\right)(\coprod i)$ and $\coprod i'$ are cofibrations.
 
 The homotopy pushout we are interested in is given by the following pushout diagram.
 \begin{center}
  \begin{tikzcd}
   \coprod S^n \arrow[swap]{d}{\coprod i'} \arrow{r}{\psi} &\mathcal{A} \arrow{d}\\
   \coprod\mathcal{T}' \arrow{r} &\mathcal{B}
  \end{tikzcd}
 \end{center}
 
 It is standard that the concatenation
 \begin{center}
  \begin{tikzcd}
   \left(\tilde{\#}\circ\tilde{c}\right)\left(\coprod S^n\right) \arrow[swap]{d}{\left(\tilde{\#}\circ\tilde{c}\right)\left(\coprod i\right)} \arrow{r}{\coprod\theta} &\coprod S^n \arrow{d}{\coprod i'} \arrow{r}{\psi} &\mathcal{A} \arrow{d}\\
   \left(\tilde{\#}\circ\tilde{c}\right)(\coprod\mathcal{T}) \arrow{r} &\coprod\mathcal{T}' \arrow{r} &\mathcal{B}
  \end{tikzcd}
 \end{center}
 of the two pushout squares is itself a pushout. 
 Also, the pushout
 \begin{center}
  \begin{tikzcd}
   \left(\tilde{\#}\circ\tilde{c}\right)\left(\coprod S^n\right) \arrow[swap]{d}{\left(\tilde{\#}\circ\tilde{c}\right)\left(\coprod i\right)} \arrow{r}{\varphi\circ\psi\circ\left(\coprod\theta\right)} & \left(\tilde{\#}\circ\tilde{c}\right)(\mathcal{A}') \arrow{d}{j}\\
   \left(\tilde{\#}\circ\tilde{c}\right)\left(\coprod\mathcal{T}\right) \arrow{r}{k} &\mathcal{B}'
  \end{tikzcd}
 \end{center}
 is entirely in the image of $\tilde{\#}\circ\tilde{c}$ as a consequence of $\tilde{\#}\circ\tilde{c}$ being a full functor (by Proposition~\ref{unit isomorphism}) and a left adjoint, hence preserving colimits.  We conclude not only that $\mathcal{B}' = \left(\tilde{\#}\circ\tilde{c}\right)(\mathcal{D})$ for some small category $\mathcal{D}$, but also that $\mathcal{D}$ is an $(n + 1)$-dimensional CW-complex.
 
 By the universal property of pushouts, there exists a unique $h$ making the diagram commute.
 \begin{center}
  \begin{tikzcd}
   \left(\tilde{\#}\circ\tilde{c}\right)\left(\coprod S^n\right) \arrow[swap]{d}{\left(\tilde{\#}\circ\tilde{c}\right)\left(\coprod i\right)} \arrow{r}{\psi\circ\left(\coprod\theta\right)} &\mathcal{A} \arrow{d} \arrow[bend left]{ddr}{j\circ\varphi}\\
   \left(\tilde{\#}\circ\tilde{c}\right)\left(\coprod\mathcal{T}\right) \arrow{r} \arrow[bend right, swap]{drr}{k} &\mathcal{B} \arrow[dashed]{dr}{h}\\
   &&\mathcal{B}'
  \end{tikzcd}
 \end{center}
 
 We now redraw this last diagram as follows.
 \begin{center}
  \begin{tikzcd}
   \left(\tilde{\#}\circ\tilde{c}\right)\left(\coprod S^n\right) \arrow[swap]{d}{\tilde{\#}\circ\tilde{c}\left(\coprod i\right)} \arrow{r}{\psi\circ\left(\coprod\theta\right)} &\mathcal{A} \arrow{d} \arrow{r}{\varphi} & \left(\tilde{\#}\circ\tilde{c}\right)(\mathcal{A}') \arrow{d}{j}\\
   \left(\tilde{\#}\circ\tilde{c}\right)\left(\coprod\mathcal{T}\right) \arrow{r} &\mathcal{B} \arrow[swap]{r}{h} &\mathcal{B}'
  \end{tikzcd}
 \end{center}
 
 The outer rectangle and left square have both been shown to be pushouts; it is classical (the dual of the ``pullback lemma'') that the right square is then also a pushout.  Finally, the middle arrow is a cofibration as it is a pushout of the cofibration $\left(\tilde{\#}\circ\tilde{c}\right)\left(\coprod i\right)$, so left-properness of $\Cat\Sh(X)$ gives $h$ as a weak equivalence as the pushout of the weak equivalence $\varphi$ along a cofibration. Now $h: \mathcal{B} \rightarrow \mathcal{B}^{\prime} = \left(\tilde{\#}\circ\tilde{c}\right)(\mathcal{D})$ is a weak equivalence from $\mathcal{B}$ to $\tilde{\#}\circ\tilde{c}$ applied to an $(n+1)$-dimensional CW-complex in $\SmCat$, completing the inductive step.
\end{proof}

\begin{theorem} \label{connected space corollary}
Let $X$ be a connected topological space. Then the homotopy category $\Ho\left(\CW\Cat\Sh(X)\right)$ of CW-complexes in internal categories in sheaves on $X$ is equivalent to $\Ho(\CW\SmCat)$, the homotopy category of CW-complexes in small categories.
\end{theorem}
\begin{proof}
For the duration of this proof, we adopt the convention 
that, given a functor $F$ between $\SmCat$ and $\Cat(\Sh(X))$ which sends $CW$-complexes to $CW$-complexes, we will write $F^{\prime}$ for its restriction to a functor between the $CW$-complex categories $\CW(\SmCat)$ and $\CW(\Cat(\Sh(X)))$.

We claim that the functor $\Ho\left((\tilde{\#}\circ \tilde{c})^{\prime}\right) : \Ho\left(\left(\CW(\SmCat)\right)\right) \rightarrow \Ho\left(\left(\CW\Cat\Sh(X)\right)\right)$ is full, faithful, and essentially surjective; clearly this implies the claim made in the statement of the theorem. We prove each of these three properties, in order:
\begin{description}
\item[Fullness] By Proposition~\ref{unit isomorphism}, $\tilde{\#}\circ\tilde{c}$ is full; so $\Ho(\tilde{\#}\circ\tilde{c})$ is also full, so $\Ho\left((\tilde{\#}\circ\tilde{c})^{\prime}\right)$ is also full.
\item[Faithfulness]   
Since $\tilde{\Gamma}\circ \tilde{\#} \circ\tilde{c} = \id$, we also have 
\begin{align}
\label{equality 0348347} \Ho\left( \tilde{\Gamma}\circ \tilde{\#} \circ\tilde{c}\right) 
 &\simeq \Ho\left( \tilde{\Gamma}\right)\circ \Ho\left(\tilde{\#} \circ\tilde{c}\right) \\
\nonumber &\simeq \Ho(\id_{\SmCat}) \\
\label{equality 0348347a} &\simeq \id_{\Ho(\SmCat)}
\end{align}
with equalities~\ref{equality 0348347} and~\ref{equality 0348347a} due to $\Ho$ being a pseudo-$2$-functor from model categories to categories, by Theorem~1.4.3 of~\cite{MR1650134}. (Here it is important that we are using the functors $\tilde{\#} \circ\tilde{c}$ and $\Gamma$, which are defined on the model categories defined by Joyal and Tierney, rather than their restrictions $(\tilde{\#} \circ\tilde{c})^{\prime}$ and $\Gamma^{\prime}$ to the $CW$-complex categories, which are not known to be model categories!)
So $\Ho\left(\tilde{\#} \circ\tilde{c}\right)$ is faithful, so its restriction
$\Ho\left((\tilde{\#} \circ\tilde{c})^{\prime}\right)$ is also faithful.
\item[Essential surjectivity] 
Immediate from Theorem~\ref{connected space}. (This is the only place in the proof of the present theorem where it is essential that we work with $(\tilde{\#}\circ \tilde{c})^{\prime}$, rather than $\tilde{\#}\circ \tilde{c}$, which is {\em not} essentially surjective!)
\end{description}
\end{proof}

\section{Formal CW-complexes in small categories} 
Next we classify CW-complexes in the category of small categories under the Joyal-Tierney model structure. 
{\em We remind the reader that our notion of CW-complexes is defined up to weak equivalence, not up to isomorphism}, as explained in the remark preceding Definition~\ref{def of spheres}.
\begin{prop}\label{spheres}
 The spheres in the category of small categories under the Joyal-Tierney model structure are, up to weak equivalence, as follows:
 \begin{align*}
  S^0 &= 1\amalg 1\\
  S^1 &\simeq \mathbb{Z},\mbox{\ \ as\ a\ one-object\ category}\\
  S^n &\simeq 1 \mbox{\ \ for\ all\ \ } n > 1
 \end{align*}
\end{prop}
\begin{proof}
By Observation~\ref{j-t is left-proper},
$\SmCat$ is left proper, implying that the homotopy pushout of the diagram
 \begin{center}
 \begin{tikzcd}
   1\amalg 1 \arrow{d}\arrow{r} &1\\
   1
  \end{tikzcd}
 \end{center}
coincides (up to weak equivalence) with the pushout of the diagram
 \begin{center}
  \begin{tikzcd}
   1\amalg 1 \arrow{d}\arrow{r} &1\\
   \mathbf{C}_2
  \end{tikzcd}
 \end{center}
 where $\mathbf{C}_2$ is the category consisting of two objects linked by a single isomorphism.  This pushout identifies the two points of $\mathbf{C}_2$ while leaving no relations on the isomorphism, yielding the integers $\mathbb{Z}$ as a 1-sphere. Suspending $\mathbb{Z}$ as a pushout is done by the diagram
 \begin{center}
  \begin{tikzcd}
   \mathbb{Z} \arrow{d}\arrow{r} &1\\
   1
  \end{tikzcd}
 \end{center}
 in which both maps are already cofibrations, implying the homotopy pushout is just the ordinary pushout, which is clearly the terminal $1$.  As the suspension of the terminal is just the terminal again, the proof is complete.
\end{proof}

\begin{theorem}\label{1-complexes}
 A small category a $1$-dimensional CW-complex, under the Joyal-Tierney model structure,
if and only if it is a groupoid with the property that the automorphism group of each object is a free group.
\end{theorem}
\begin{proof}
 Clearly a $0$-complex is just a disjoint union of terminals $\coprod 1$, ie a discrete category.  A $1$-complex is then built on a $0$-complex $Z_0 := \coprod 1$ by the homotopy pushout of
 \begin{center}
  \begin{tikzcd}
   \coprod_IS^0 \arrow{d}\arrow{r} &Z_0\\
   \coprod_I1
  \end{tikzcd}
 \end{center}
 where I is an indexing set.  By replacing the vertical map by a cofibration\footnote{I.e., by factoring the left vertical map as a cofibration followed by an acyclic fibration, and using the cofibration in our pushout diagram rather than the original vertical map. There are other choices of cofibration-acyclic-fibration factorization of the vertical map other than the one shown in diagram~\eqref{pushout diag 132084}, but it is standard that making another other choice of factorization results in a weak equivalent pushout.}., we get
 \begin{equation}\label{pushout diag 132084}
  \begin{tikzcd}
   \coprod_IS^0 \arrow{d}\arrow{r} &Z_0\\
   \coprod_I\mathbf{C}_2.
  \end{tikzcd}
 \end{equation}

 The pushout of this diagram simply attaches isomorphisms to our discrete $0$-complex without imposing any rules on how they must compose.  That is, each $1$-complex is 
a groupoid in which each object's automorphism group is free. 

Conversely, suppose we have such a groupoid $\mathbf{G}$.  We claim it is possible to construct $\mathbf{G}$ as a $1$-complex. First, let $\mathbf{G}_i'$ be the full subcategory of $\mathbf{G}$ given by a single object $G_i$ of $\mathbf{G}$, so $\mathbf{G}_i'$ is a free group on some generating set $S$.  We let $Z_0$ be the terminal $1$ and take the pushout of the diagram
 \begin{center}
  \begin{tikzcd}
   \coprod_SS^0 \arrow{d}\arrow{r} &1\\
   \coprod_S\mathbf{C}_2
  \end{tikzcd}
 \end{center}

 This attaches to a single object a set of isomorphisms indexed by $S$ with no relations they much satisfy.  Clearly this is the free group on $S$.

 Next, suppose $\mathbf{G}_i$ is the full {\em replete} subcategory of $\mathbf{G}$ generated by the object $G_i$, that is, $\mathbf{G}_i$ is the full subcategory of $\mathbf{G}$ given by the objects isomorphic to $G_i$. Index all such objects excluding $G_i$ itself by a set $T$.  Viewing $T$ as a discrete category (a $0$-complex), let $q:  T \to 1\amalg T$ send every object to the terminal $1$, and let $\iota$ be the insertion of $T$ into $1\amalg T$.  We define a functor $r$ of discrete categories by the coproduct diagram
 \begin{center}
  \begin{tikzcd}
   T \arrow[swap]{dr}{q} \arrow[hook]{r} &\coprod_T S^0 \arrow[dashed]{d}{r} &T \arrow[hook]{l} \arrow{dl}{\iota}\\
   &1\amalg T,
  \end{tikzcd}
 \end{center}
 so that $r$ sends the second object of each copy of $S^0$ to its corresponding object under $T$, and the first always goes to the remaining terminal.  Now let $p': \coprod_S S^0 \to 1\amalg T$ be functor sending everything to $1$, and define $p: \coprod_{S\amalg T}S^0 \to 1\amalg T$ by
 \begin{center}
  \begin{tikzcd}
   \coprod_S S^0 \arrow[swap]{dr}{p'} \arrow[hook]{r} &\coprod_{S\amalg T}S^0 \arrow[dashed]{d}{p} &\coprod_T S^0 \arrow[hook]{l} \arrow{dl}{r}\\
   &1\amalg T
  \end{tikzcd}
 \end{center}

 Using this map, the pushout of
 \begin{center}
  \begin{tikzcd}
   \coprod_{S\amalg T}S^0 \arrow{d}\arrow{r}{p} &1\amalg T\\
   \coprod_{S\amalg T}\mathbf{C}_2
  \end{tikzcd}
 \end{center}
 glues isomorphisms indexed by $S$ to the same object $1$, again making the object's automorphism group free on $S$, but now also glues from each object of $T$ a single isomorphism to $1$. 
 The facts that $\mathbf{G}$ is simply a coproduct of all such isomorphism components $\mathbf{G}_i$, and that coproducts commute with pushouts complete the proof.
\end{proof}

Higher-dimensional complexes are even easier to characterize---they are exactly the groupoids!
\begin{theorem}\label{2-complexes}
 A small category is a $2$-dimensional CW-complex, under the Joyal-Tierney model structure, if and only if it is a groupoid.
\end{theorem}
\begin{proof}
 Taking a pushout
 \begin{center}
  \begin{tikzcd}
   \coprod_I \mathbb{Z} \arrow{d} \arrow{r} &\mathbf{G}\\
   \coprod_I 1
  \end{tikzcd}
 \end{center}
 where $I$ is an index set and $\mathbf{G}$ is a groupoid will not introduce any morphisms which are not invertible, hence every $2$-dimensional CW-complex is weakly equivalent to a groupoid, hence is a groupoid.
 
 To show the other direction, let $\mathbf{G}$ be a groupoid, and let $G_i$ be an object in $\mathbf{G}$.  Its automorphism group, like any group is a quotient of a free group on some set $S$ by some subgroup $G_S$.  Letting the set of objects isomorphic to $G_i$ be indexed by a set $T$, construct the $1$-complex $\mathbf{T}_S$ on $T$ where all pairs of objects are connected by an isomorphism and every automorphism group is free on $S$.  Let $p: \coprod_J\mathbb{Z} \to \mathbf{T}_S$ be a functor sending all objects in the domain to the same object in the codomain, while sending the generators of each copy of $\mathbb{Z}$ to a generating set of $G_S$.  The pushout of
 \begin{center}
  \begin{tikzcd}
   \coprod_J\mathbb{Z} \arrow{d} \arrow{r}{p} &\mathbf{T}_S\\
   \coprod_J 1
  \end{tikzcd}
 \end{center}
 collapses each automorphism group to that of $G_i$.  Because coproducts commute with pushouts, the result follows.
\end{proof}

Since $S^n$ is contractible for all $n>2$, we have the following:
\begin{theorem}\label{n-complexes}
 A small category is weakly equivalent to a CW-complex, under the Joyal-Tierney model structure, if and only if it is a groupoid.
\end{theorem}

\begin{corollary}\label{main cor on cw-cplxs 1} 
The category $\CW \SmCat$ of CW-complexes in small categories is the full subcategory of $\SmCat$ generated by the groupoids.
\end{corollary}

\begin{theorem}\label{main thm on cw-cplxs}
 An internal category in $\Sh(X)$ for a connected space $X$ is weakly equivalent to a 
CW-complex, under the Joyal-Tierney model structure, if and only if it is weakly equivalent to a sheafification of a constant groupoid presheaf.
\end{theorem}
\begin{proof}
By Theorems~\ref{connected space} and~\ref{n-complexes}.
\end{proof}

\begin{corollary}\label{main cor on cw-cplxs 2}
Let $X$ be a connected topological space. The homotopy category $\Ho \left( \CW\Cat\Sh(X)\right)$ of CW-complexes in internal categories in the category of sheaves on $X$ is equivalent to the homotopy category of groupoids.
\end{corollary}
\begin{proof}
Immediate from Theorem~\ref{connected space corollary} and Corollary~\ref{main cor on cw-cplxs 1}.
\end{proof}

\section{Algebraic $K$-theory.}

\subsection{Generalities on algebraic $K$-theory.}

Let $\mathcal{C},\mathcal{D}$ be Waldhausen categories, that is,
$\mathcal{C}$ and $\mathcal{D}$ are ``categories with cofibrations and weak equivalences''
in the sense of~\cite{MR802796}.
Let $F,G,H: \mathcal{C}\rightarrow\mathcal{D}$ be exact functors,
that is,
$F$ is a functor which sends the zero object to the zero object, sends cofibrations to cofibrations, sends weak equivalences to weak equivalences, and sends pushouts along cofibrations to pushouts, and similarly for $G$ and $H$. 
Suppose that we have natural transformations $F\rightarrow G$ and $G\rightarrow H$ such that, for each object $X$ of $\mathcal{C}$,
the resulting sequence
\[ F(X) \rightarrow G(X) \rightarrow H(X)\]
is a cofiber sequence in $\mathcal{C}$.
Finally, suppose that, for every cofibration
$X \rightarrow Y$ in $\mathcal{C}$, the resulting 
map $G(X) \coprod_{F(X)} F(Y)\rightarrow G(Y)$ is also a cofibration.
Then we have Waldhausen's famous additivity theorem, from~\cite{MR802796}:
\begin{theorem} {\bf (Waldhausen's additivity theorem.)}\label{additivity thm}
The functions
$\left| wS.G\right|$ and $\left| wS.( F\vee H)\right|$,
from $\left|wS.\mathcal{C}\right|$ to $\left|wS.\mathcal{D}\right|$, are homotopic.\end{theorem}
Consequently the maps induced in $K$-theory groups,
$K_*(G)$ and $K_*(F\vee H)$, are equal.
(Waldhausen gives several equivalent versions of the additivity
theorem in~\cite{MR802796}; the version stated in Theorem~\ref{additivity thm} is simply the version which is most convenient for the purposes of the present paper.)

There is another result from Waldhausen's paper~\cite{MR802796} which we will use. This is Proposition~1.3.1 from that paper:
\begin{prop} \label{waldhausens homotopies}
Let $\mathcal{C},\mathcal{D}$ be Waldhausen categories,
and let $F,G: \mathcal{C}\rightarrow\mathcal{D}$
be exact functors such that, for every object $X$ of $\mathcal{C}$,
the induced map $F(X) \rightarrow G(X)$ is a weak equivalence in $\mathcal{D}$. Then the
functions $\left| wS.F\right|$ and $\left| wS.G\right|$,
from $\left| wS.\mathcal{C} \right|$ to $\left| wS.\mathcal{D} \right|$, are homotopic.
\end{prop}

\begin{definition}
Let $\mathcal{C}$ be a Waldhausen category. We will say that an object $X$ of $\mathcal{C}$ is {\em contractible} if the unique map
$0 \rightarrow X$ is a weak equivalence.
\end{definition}

\begin{definition}\label{def of cone functor}
Let $\mathcal{C}$ be a Waldhausen category.
We will say that an exact functor $P: \mathcal{C}\rightarrow \mathcal{C}$ equipped with a natural transformation $\eta: \id_{\mathcal{C}}\rightarrow P$ is a {\em cone functor}
if:
\begin{itemize}
\item for all objects $X$ of $\mathcal{C}$, $PX$ is contractible, 
\item for all objects $X$ of $\mathcal{C}$, the map $\eta(X) : X \rightarrow PX$ is a cofibration, and
\item for every cofibration $X \rightarrow Y$ in $\mathcal{C}$, the map
$Y\coprod_{X}PX \rightarrow PY$ is a cofibration.
\end{itemize}

If $\mathcal{C}$ is a Waldhausen category and $(P, \eta)$ is a cone functor on $\mathcal{C}$,
then we will write $\Sigma: \mathcal{C} \rightarrow \mathcal{C}$ 
for the functor given by $\Sigma(X) = PX \coprod_X 0$,
i.e., $\Sigma(X)$ is the pushout:
\[\xymatrix{
 X \ar[r]^{\eta(X)} \ar[d] & PX \ar[d] \\ 0 \ar[r] & \Sigma(X) .}\]
When $n$ is a positive integer, we will write $\Sigma^n$
for the $n$-fold composite of $\Sigma$ with itself.
\end{definition}

It is standard that, given a pointed model category, the full subcategory generated by its cofibrant objects forms a Waldhausen category.
The one nontrivial part of the proof of this fact is checking
Waldhausen's condition ``{\bf Weq 2}'' from~\cite{MR802796},
which
reads as follows: if we have a commutative diagram
\[\xymatrix{ 
 B \ar[d] & A \ar[l]^f\ar[r]\ar[d] & C \ar[d] \\
 B^{\prime} & A^{\prime} \ar[l]^g\ar[r] & C^{\prime} }\]
in $\mathcal{C}$ in which $f$ and $g$ are cofibrations
and all three vertical arrows are weak equivalences,
then the induced map
\[ B\coprod_A C \rightarrow B^{\prime} \coprod_{A^{\prime}} C^{\prime}\]
is a weak equivalence.
If $\mathcal{A}$ is a pointed model category and
$\mathcal{C}$ is the full subcategory of $\mathcal{A}$ generated
by the cofibrant objects, then condition {\bf Weq 2} holds in $\mathcal{C}$ by Lemma~5.2.6 in M. Hovey's book~\cite{MR1650134} (Hovey, in turn, writes that he learned the result from Dwyer, Hirschhorn, and Kan); the same result also appears as the corollary following Theorem~B in~\cite{reedy}.

Now here is an easy consequence of Theorem~\ref{additivity thm} and Proposition~\ref{waldhausens homotopies}:
\begin{prop}\label{k-thy vanishing thm}
Let $\mathcal{C}$ be a Waldhausen category.
Suppose that there exists a cone functor on $\mathcal{C}$
and a positive integer $n$ such that, for all objects $X$
of $\mathcal{C}$, $\Sigma^n(X)$ is contractible.
Then the space $\left| wS.\mathcal{C}\right|$ is contractible.
Consequently the $K$-theory groups $K_n(\mathcal{C})$ are trivial for all $n\geq 0$.
\end{prop}
\begin{proof}
This line of argument is quite standard.
If $F: \mathcal{C}\rightarrow \mathcal{C}$ is an exact functor,
then we have the cofiber sequences of exact functors
\begin{align*}
 F &\rightarrow PF \rightarrow \Sigma F,\\
 \Sigma F &\rightarrow P\Sigma F \rightarrow \Sigma^2 F,\end{align*}
and then Theorem~\ref{additivity thm} yields 
the homotopies
\begin{align*}
 PF \vee \Sigma^2 F 
  &\simeq F \vee \Sigma F \vee \Sigma^2 F \\
  &\simeq F \vee P\Sigma F.
\end{align*}
Since $P\Sigma F(X)$ and $PF(X)$ are contractible for all $X$,
Proposition~\ref{waldhausens homotopies} gives us further homotopies:
\begin{align*}
 \Sigma^2 F &\simeq  PF \vee \Sigma^2 F \\
  &\simeq F \vee P\Sigma F \\
  &\simeq F.
\end{align*}
Now we apply the homotopy $\Sigma^2 F \simeq F$ repeatedly, letting $F = \Sigma^{2i}$ for $i=0, 2, \dots $,
to get homotopies
\begin{equation}\label{chain of homotopies} \id \simeq \Sigma^2 \simeq \Sigma^4 \simeq \Sigma^6 \simeq \dots .\end{equation}
By assumption, there exists some $n$ such that $\Sigma^n(X)$ is contractible for all objects $X$ of $\mathcal{C}$,
so let $m$ be any even integer greater than or equal to $n$, and now $\Sigma^m$ is homotopic to the constant basepoint-valued 
endomorphism of $\left| wS.\mathcal{C}\right|$, by Proposition~\ref{waldhausens homotopies}.
So $\Sigma^m$ is nulhomotopic, so the identity map of 
$\left| wS.\mathcal{C}\right|$ is nulhomotopic
due to the chain of homotopies~\eqref{chain of homotopies}.
Since the identity map of $\left| wS.\mathcal{C}\right|$ is nulhomotopic, $\left| wS.\mathcal{C}\right|$ is contractible as claimed.
\end{proof}
While Proposition~\ref{k-thy vanishing thm} is easy to prove (once one knows Waldhausen's additivity theorem), our original proof
used a bi-exact product on $\mathcal{C}$, and multiplicative structure on the $K$-theory space to show its contractibility; we are
grateful to J. Klein for suggesting to us that we try to prove the same result without the multiplicative structure,
by using the additivity theorem.

\subsection{Application to the Joyal-Tierney model structure.}

Now we are going to do some algebraic $K$-theory with internal categories
in a Grothendieck topos. In order to get a Waldhausen category
of internal categories in a Grothendieck topos $\Sh(\mathcal{C},\tau)$, we need to have
a {\em pointed} category of such internal categories; to accomplish this,
we simply take the category of pointed objects in $\Cat\Sh(\mathcal{C},\tau)$, i.e., the undercategory consisting of internal categories $\mathcal{X}$ in $\Sh(\mathcal{C},\tau)$ equipped with a choice of map
from the terminal internal category $1$ to $\mathcal{X}$.
We will focus on the most important special case, the case $\Sh(\mathcal{C},\tau) = \SmCat$.
Following~\cite{MR1650134}, we will write $\SmCat_*$ for the category of pointed objects in small categories. By Proposition~1.1.8 of~\cite{MR1650134},
$\SmCat_*$ admits a model structure whose cofibrations
(respectively, weak equivalences, fibrations) are exactly those maps
whose underlying maps in $\SmCat$
are cofibrations (respectively, weak equivalences, fibrations).
We will call this model structure on $\SmCat_*$
the {\em pointed Joyal-Tierney model structure.}

\begin{prop}\label{existence of cone functor}
The category $\SmCat_*$, equipped with the Waldhausen category structure it inherits from the pointed Joyal-Tierney model structure, admits a cone functor in the sense of Definition~\ref{def of cone functor}.
\end{prop}
\begin{proof}
We give the details explicitly, below, but the intuition here is that the cone functor $P: \SmCat_* \rightarrow \SmCat_*$ wlil be simply the functor whose value on given a small category $\mathcal{C}$ is the category with the same object set as $\mathcal{C}$ and which has, for each pair of objects $X,Y$, a unique morphism $X \rightarrow Y$.

Now we give the details. Let $P: \SmCat_* \rightarrow \SmCat_*$
be the functor defined as follows:
if $\mathcal{X}$ is a pointed small category,
and we agree to write $\mathcal{X}_{ob}$ for the object set of $\mathcal{X}$ and $\mathcal{X}_{mor}$ for the morphism set of $\mathcal{X}$,
$\eta_L,\eta_R: \mathcal{X}_{mor}\rightarrow\mathcal{X}_{ob}$ for the domain and codomain morphisms, 
$\epsilon: \mathcal{X}_{ob}\rightarrow \mathcal{X}_{mor}$ for the identity morphism morphism, 
and $\Delta: \mathcal{X}_{mor}\times_{\mathcal{X}_{ob}}\mathcal{X}_{mor} \rightarrow \mathcal{X}_{mor}$ for the composition of morphisms morphism,
then let $P(\mathcal{X})$ be the pointed internal category with $P(\mathcal{X})_{ob} = \mathcal{X}_{ob}$ and $P(\mathcal{X})_{mor} = \mathcal{X}_{ob}\times \mathcal{X}_{ob}$,
with structure maps as follows:
\begin{itemize}
\item the morphism $P(\mathcal{X})_{mor}\rightarrow P(\mathcal{X})_{ob}$, sending a morphism to its domain, is the projection to the first factor
$\mathcal{X}_{ob}\times \mathcal{X}_{ob} \rightarrow \mathcal{X}_{ob}$,
\item the morphism $P(\mathcal{X})_{mor}\rightarrow P(\mathcal{X})_{ob}$, sending a morphism to its codomain, is the projection to the second factor
$\mathcal{X}_{ob}\times \mathcal{X}_{ob} \rightarrow \mathcal{X}_{ob}$,
\item the morphism $P(\mathcal{X})_{ob}\rightarrow P(\mathcal{X})_{mor}$, sending an object to its identity morphism, 
is the diagonal morphism $\mathcal{X}_{ob} \rightarrow \mathcal{X}_{ob}\times \mathcal{X}_{ob}$,
\item and the morphism $P(\mathcal{X})_{mor}\times_{P(\mathcal{X})_{ob}}\rightarrow P(\mathcal{X})_{mor}$, sending a composable pairs of morphisms to their composite,
is the composite morphism 
\[ (\mathcal{X}_{ob}\times \mathcal{X}_{ob})\times_{\mathcal{X}_{ob}} (\mathcal{X}_{ob}\times \mathcal{X}_{ob})
 \stackrel{\cong}{\longrightarrow} \mathcal{X}_{ob}\times \mathcal{X}_{ob} \times \mathcal{X}_{ob}
 \stackrel{r}{\longrightarrow} \mathcal{X}_{ob}\times \mathcal{X}_{ob}\]
where $r$ is the projection off of the middle factor (i.e., the projection on to the left and right factors).
\end{itemize}
There exists a natural transformation $\id\stackrel{\eta}{\longrightarrow} P$ of functors 
$\SmCat \rightarrow \SmCat$
defined as follows:
given a pointed small category $\mathcal{X}$,
the map $\eta(\mathcal{X}): \mathcal{X}\rightarrow P(\mathcal{X})$
has components
$\eta(\mathcal{X})_{ob}: \mathcal{X}_{ob}\rightarrow P(\mathcal{X})_{ob}$
simply 
the identity map,
and 
$\eta(\mathcal{X})_{mor}: \mathcal{X}_{mor}\rightarrow P(\mathcal{X})_{mor}$
the map $\mathcal{X}_{mor} \rightarrow \mathcal{X}_{ob}\times \mathcal{X}_{ob}$
given by the universal property of the product in $\SmCat$
from the maps $\eta_L,\eta_R: \mathcal{X}_{mor} \rightarrow \mathcal{X}_{ob}$.
It is a routine exercise to check that $P$ is indeed a functor and that
$\eta$ is a natural transformation.

We claim that $(P,\eta)$ is a cone functor on $\SmCat_*$ in the pointed Joyal-Tierney model structure.
Most of this claim is basically trivial to prove:
clearly $\eta(\mathcal{X})$ is a cofibration for each 
small category $\mathcal{X}$,
and clearly $P$ sends cofibrations to cofibrations.
The unique map from $P(\mathcal{X})$ to the terminal internal category $1$
is clearly full and faithful (as in Definition~\ref{model structure}),
and it is essentially surjective as well (the composite map
$\eta_R\circ \iota\circ p$ in Definition~\ref{model structure} is,
after an isomorphism, simply a projection map
$\mathcal{X}_{ob}\times \mathcal{X}_{ob}\rightarrow \mathcal{X}_{ob}$ to a factor), so $P(\mathcal{X})$ is contractible.

We need to know that, 
if $\mathcal{X}\stackrel{f}{\longrightarrow} \mathcal{Y}$
is a weak equivalence in $\SmCat_*$,
then $P(f)$ is a weak equivalence.
In fact, more is true: $P(f)$ is a weak equivalence, even if $f$ is not!
This is simply because we have the commutative diagram
\[ \xymatrix{ P\mathcal{X} \ar[r]^{P(f)} \ar[d]^{we} & P\mathcal{Y}\ar[ld]^{we} \\ 1 & }\]
in which the two maps marked $we$ are weak equivalences, and now
the two-out-of-three property of weak equivalences in a model category tells us that $P(f)$ is also a weak equivalence.

If $\mathcal{X}\rightarrow \mathcal{Y}$
is a cofibration in $\SmCat_*$,
then we need to know that the map $\mathcal{Y}\coprod_{\mathcal{X}} P(\mathcal{X})\rightarrow P(\mathcal{Y})$ is a cofibration as well.
This follows, however, from the observation that the functor
\begin{align*}
 U: \SmCat_* &\rightarrow \Sets_* \\
 U(\mathcal{X}) & =  \mathcal{X}_{ob} \end{align*}
has a right adjoint, given by sending a pointed set $Y$
to the pointed small category
whose object set is $Y$, whose morphism set is $Y$,
and whose structure maps are all the identity map on $Y$.
Hence $U$ preserves colimits,
in particular, pushouts.
Hence the fact that $\eta(\mathcal{X}): \mathcal{X}\rightarrow P(\mathcal{X})$ induces an isomorphism 
$\eta(\mathcal{X})_{ob}: \mathcal{X}_{ob}\stackrel{\cong}{\longrightarrow} P(\mathcal{X})_{ob}$ in $\SmCat_*$
gives us that
$\left(\mathcal{Y}\coprod_{\mathcal{X}} P(\mathcal{X})\right)_{ob}\rightarrow P(\mathcal{Y})_{ob}$
is a monomorphism, hence that 
$\mathcal{Y}\coprod_{\mathcal{X}} P(\mathcal{X})\rightarrow P(\mathcal{Y})$ is a cofibration.

All that remains is to check that $P$ preserves pushouts along
a cofibration. Let $\mathcal{X},\mathcal{Y},\mathcal{Z}$ be 
pointed small categories,
and let $\mathcal{X}\stackrel{}{\longrightarrow} \mathcal{Y}$
and $\mathcal{X}\stackrel{}{\longrightarrow} \mathcal{Z}$
be morphisms in $\SmCat_*$.
Since $P(\mathcal{Y}\coprod_{\mathcal{X}} \mathcal{Z})$ and
$P\mathcal{Y}\coprod_{P\mathcal{X}} P\mathcal{Z}$
depend only on 
$(\mathcal{Y}\coprod_{\mathcal{X}} \mathcal{Z})_{ob}$
and on 
$\mathcal{Y}_{ob},\mathcal{X}_{ob},$ and $\mathcal{Z}_{ob}$,
and since the functor $U: \SmCat_* \rightarrow \Sets_*$, defined above, preserves pushouts,
to show that the canonical map 
\[ P\mathcal{Y}\coprod_{P\mathcal{X}} P\mathcal{Z}
 \rightarrow P(\mathcal{Y}\coprod_{\mathcal{X}} \mathcal{Z})\]
is an isomorphism in $\SmCat_*$ is equivalent
to simply showing that, if
$X,Y,Z$ are objects of $\Sets_*$
and $ f: X \rightarrow Y$ and $g: X \rightarrow Z$
are morphisms in $\Sets_*$
with $f$ a monomorphism,
then the canonical map
\begin{equation}\label{comparison map 22} \mathcal{F}(Y) \coprod_{\mathcal{F}(X)} \mathcal{F}(Z) 
 \stackrel{\gamma}{\longrightarrow} \mathcal{F}(Y\coprod_X Z)\end{equation}
in $\SmCat_*$ is an isomorphism,
where $\mathcal{F}: \Sets_* \rightarrow \SmCat_*$ is the functor that sends an object $W$ to 
the internal category $(W, W\times W)$, so that
$P$ factors as $P = \mathcal{F}\circ U$. 

We will now show that the map~\eqref{comparison map 22} is an isomorphism
by producing an inverse map $\delta$ to $\gamma$.
To specify such a map $\delta$, it suffices to give a natural transformation
from $\hom_{\SmCat_*}(\mathcal{F}(Y)\coprod_{\mathcal{F}(X)} \mathcal{F}(Z), -)$
to
$\hom_{\SmCat_*}(\mathcal{F}(Y\coprod_X Z), -)$
which is inverse to the natural transformation corepresented by $\gamma$.
Suppose that $\mathcal{T}$ is an object of 
$\SmCat_*$, and suppose we are given a 
morphism 
$\mathcal{F}(Y)\coprod_{\mathcal{F}(X)} \mathcal{F}(Z) \rightarrow \mathcal{T}$
in $\SmCat_*$, i.e., 
we are given 
a pair of morphisms
$a: Y \rightarrow \mathcal{T}_{ob}$
and 
$b: Z \rightarrow \mathcal{T}_{ob}$
in $\Sets_*$ which agree on $X$,
and a compatible pair of morphisms
$f: Y\times Y \rightarrow \mathcal{T}_{mor}$
and 
$g: Z\times Z \rightarrow \mathcal{T}_{mor}$
in $\Sets_*$ which agree on $X\times X$.
From this data we need to construct a morphism
$\mathcal{F}(Y\coprod_X Z) \rightarrow \mathcal{T}$.
On objects, 
\[ \left(\mathcal{F}(Y\coprod_X Z)\right)_{ob} = Y\coprod_X Z 
\rightarrow \mathcal{T}_{ob}\]
is simply given by $a$ and $b$ and the universal property of
the pushout.
On morphisms
\[ \left(\mathcal{F}(Y\coprod_X Z)\right)_{mor} = 
  \left( Y\coprod_X Z \right) \times \left( Y\coprod_X Z \right)
 \rightarrow \mathcal{T}_{mor}\]
is given by sending $(y_1,y_2)\in Y\times Y$ to
$f(y_1,y_2)$, sending 
$(z_1,z_2)\in Z\times Z$ to
$g(z_1,z_2)$, sending $(y,z)\in Y\times Z$
to the composite $g(x,z) \circ f(y,x)$ where $x$ is any element of $X$ (such an element exists: $X$ is necessarily nonempty since it is pointed), 
and by sending $(z,y)\in Z\times Y$
to the composite $f(x,y) \circ g(z,x)$ where $x$ is any element of $X$.
Checking this is well-defined is left as an exercise, but we do offer an explanation of the idea behind this construction:
if $W$ is a set, then $\mathcal{F}(W)$ is the small category whose set of objects is $W$, and which has, for every pair of objects
$w,w^{\prime}\in W$, exactly one morphism $w\rightarrow w^{\prime}$. If we have some small category $\mathcal{T}$ and 
chosen functors $f: \mathcal{F}(Y) \rightarrow \mathcal{T}$ and $g: \mathcal{F}(Z) \rightarrow \mathcal{T}$ which coincide
on $\mathcal{F}(X)$, and we want to use this information to specify a functor $\mathcal{F}(Y\coprod_X Z) \rightarrow \mathcal{T}$,
then it is clear that this functor should coincide with $f$ on $\mathcal{F}(Y)\subseteq \mathcal{F}(Y\coprod_X Z)$,
and that this functor should coincide with $g$ on $\mathcal{F}(Z)\subseteq \mathcal{F}(Y\coprod_X Z)$.
This determines everything about the behavior of the functor $\mathcal{F}(Y\coprod_X Z) \rightarrow \mathcal{T}$ {\em except}
for its effect on morphisms $y\rightarrow z$ and morphisms $z\rightarrow y$ in $\mathcal{F}(Y\coprod_X Z)$,
where $y\in Y$ and $z\in Z$. But there is a natural way to define our functor on these morphisms: 
given the (unique) morphism $h: y\rightarrow z$ in $\mathcal{F}(Y\coprod_X Z)$,
choose an element $x\in X$, and now 
in $\mathcal{F}(Y\coprod_X Z)$ there is a unique morphism $h_1: y\rightarrow x$ and a unique morphism $h_2: x\rightarrow z$,
and $h_2\circ h_1  = h$. Clearly we should let our desired functor 
$\mathcal{F}(Y\coprod_X Z) \rightarrow \mathcal{T}$ take the value $g(h_2)\circ f(h_1)$ on the morphism $h$,
and the uniqueness of morphisms with a given domain and codomain in $\mathcal{F}(Y\coprod_X Z)$ ensures that this definition
does not depend on the choice of $x$. With the obvious modification (switching the roles of $Y$ and $Z$)
this same construction defines
the functor $\mathcal{F}(Y\coprod_X Z) \rightarrow \mathcal{T}$ on the morphisms $z\rightarrow y$ as well.

We have now defined a function
\begin{equation}\label{nat trans 1}\hom_{\SmCat_*}(\mathcal{F}(Y)\coprod_{\mathcal{F}(X)} \mathcal{F}(Z), \mathcal{T}) \rightarrow \hom_{\SmCat_*}(\mathcal{F}(Y\coprod_X Z), \mathcal{T})\end{equation}
for each object $\mathcal{T}$ in $\SmCat_*$, and function~\eqref{nat trans 1} is natural in the variable
$\mathcal{T}$.
Hence~\eqref{nat trans 1} defines a morphism $\mathcal{F}(Y\coprod_X Z) \rightarrow \mathcal{F}(Y)\coprod_{\mathcal{F}(X)} \mathcal{F}(Z)$.
It is routine to verify that this morphism is inverse to the canonical morphism $\gamma$ from~\eqref{comparison map 22}.

Hence the map $\gamma$ is an isomorphism,
hence $\mathcal{F}$ preserves pushouts, as desired.
Hence $P$ preserves pushouts, and in particular, pushouts along a cofibration.
\end{proof}

\begin{prop}\label{double suspension is contractible}
For each object $X$ of $\SmCat_*$, the second suspension $\Sigma^2 X$ of $X$ is weakly equivalent to the terminal object in $\SmCat_*$.
\end{prop}
\begin{proof}
The suspension $\Sigma X$ is, up to weak equivalence, the pushout of the diagram
\begin{equation}\label{htpy pushout diagram 1}\xymatrix{
 X \ar[r]^{\eta(X)} \ar[d] & PX \\ 1 & 
}\end{equation}
where $PX$ is as in Proposition~\ref{existence of cone functor}.
(One can take this statement in either of two ways: either by the definition of $\Sigma$ as in Definition~\ref{def of cone functor},
$\Sigma X$ is the pushout of the diagram~\eqref{htpy pushout diagram 1}; or, by the definition of suspension in a pointed model category,
$\Sigma X$ is the homotopy pushout of the diagram
\begin{equation*}
\xymatrix{
 X \ar[r]^{} \ar[d] & 1 \\ 1 .& 
}\end{equation*}
Since every object in the Joyal-Tierney model structure on $\SmCat$ is cofibrant, pushouts of weak equivalences along
cofibrations in $\SmCat$ are weak equivalences, by Theorem~B in~\cite{reedy}. 
Hence $\SmCat$ is left proper, hence the pushout of the diagram~\eqref{htpy pushout diagram 1} is a homotopy pushout,
since $\eta(X)$ is a cofibration. So whether one takes $\Sigma X$ to mean suspension in the sense of Definition~\ref{def of cone functor}, or
suspension in the sense of the suspension defined in any model category, in either case the pushout of diagram~\eqref{htpy pushout diagram 1} is 
$\Sigma X$, at least up to weak equivalence.)

Now here are two observations about the pushout of the diagram~\eqref{htpy pushout diagram 1}:
\begin{enumerate}
\item The pushout $\Sigma X$ has the property that $(\Sigma X)_{ob} \cong 1$,
the terminal object in $\SmCat$, since the functor $W\mapsto W_{ob}: \SmCat_* \rightarrow \Sets_*$ preserves pushouts and since $(\eta(X))_{ob}$ is an isomorphism.
\item If $X_{ob} \cong 1$, then $PX$ is the terminal object in $\SmCat_*$, so $\Sigma X$ is a quotient of the terminal object,
i.e., $\Sigma X$ is terminal.
\end{enumerate}
Consequently $\Sigma(\Sigma X)$ is terminal in $\SmCat_*$.
\end{proof}

\begin{theorem}\label{k-thy vanishing thm 2}
Equip $\SmCat_*$ with the Waldhausen category structure inherited from the Joyal-Tierney model structure.
Let $\mathcal{A}$ be any full sub-Waldhausen-category of $\SmCat_*$ satisfying the following conditions:
\begin{itemize}
\item pushouts along cofibrations computed in $\mathcal{A}$ agree with the same pushouts computed in $\SmCat_*$, and
\item if $(\mathcal{T}_{ob},\mathcal{T}_{mor})$ is an object in $\mathcal{A}$, then so is $(\mathcal{T}_{ob}, \mathcal{T}_{ob}\times \mathcal{T}_{ob})$, with structure maps defined as in the proof of Proposition~\ref{existence of cone functor}.
\end{itemize}

Then the algebraic $K$-theory space $\left| wS.\mathcal{A}\right|$ of $\mathcal{A}$ is contractible, and hence the algebraic $K$-theory groups
$K_n(\mathcal{A})$ are trivial for all $n\geq 0$.
\end{theorem}
\begin{proof}
Immediate consequence of Proposition~\ref{k-thy vanishing thm}, Proposition~\ref{existence of cone functor}, and Proposition~\ref{double suspension is contractible}.
\end{proof}

\bibliography{salch}{}
\bibliographystyle{plain}
\end{document}